\documentclass[10pt]{amsart}
\usepackage{setspace,xypic,somedefs,tracefnt,latexsym,rawfonts,graphicx,amsfonts,amssymb,hyphenat,latexsym,enumerate,amscd}
\makeatletter
\@namedef{subjclassname@2020}{%
  \textup{2020} Mathematics Subject Classification}
\makeatother
\def\cal{\mathcal}
\def\Bbb{\mathbb}

\def\wt{\widetilde}

\newtheorem{thm}{Theorem}[section]
\newtheorem{prop}[thm]{Proposition}
\newtheorem{exm}[thm]{Example}
\newtheorem{lemma}[thm]{Lemma}
\newtheorem{cor}[thm]{Corollary}
\newtheorem{defn}[thm]{Definition}
\newtheorem{rem}[thm]{Remark}

\numberwithin{equation}{section}
\makeatletter
\newcommand{\colim@}[2]{%
  \vtop{\m@th\ialign{##\cr
    \hfil$#1\operator@font colim$\hfil\cr
    \noalign{\nointerlineskip\kern1.5\ex@}#2\cr
    \noalign{\nointerlineskip\kern-\ex@}\cr}}%
}
\newcommand{\colim}{%
  \mathop{\mathpalette\colim@{\rightarrowfill@\textstyle}}\nmlimits@
}
\makeatother
\begin{document}
\date{\today}
\title[On aspherical configuration Lie groupoids]{On aspherical configuration
  Lie groupoids}
\author[S.K. Roushon]{S.K. Roushon}
\address{School of Mathematics\\
Tata Institute\\
Homi Bhabha Road\\
Mumbai 400005, India}
\email{roushon@math.tifr.res.in} 
\urladdr{https://mathweb.tifr.res.in/\~\!\!\! roushon/}
\begin{abstract} The complement of the hyperplanes
  $\{x_i=x_j\}$, for all $i\neq j$ in $M^n$, for $M$ an aspherical $2$-manifold,
  is known to be aspherical. Here we consider the situation, when
  $M$ is
  a $2$-dimensional orbifold. We prove this complement to be aspherical for a
  class of aspherical $2$-dimensional orbifolds, and predict that
  it should be true in general also.
  We generalize this question in the
  category of Lie groupoids, as orbifolds can be identified with a certain
  kind of Lie groupoids.
  \end{abstract}
 

\maketitle
\tableofcontents

\section{Introduction}
Let $X$ be a topological space and 
$PB_n(X)$ be the configuration space of ordered
$n$-tuples of pairwise distinct points of $X$. That is,
$$PB_n(X)=\{(x_1,x_2,\dots, x_n)\in X^n\ |\ x_i\neq x_j,\ \text{for}\ 
i\neq j\}.$$

Let $M$ be a connected manifold of dimension $\geq 2$, and $n\geq 2$.
The Fadell-Neuwirth fibration theorem ([\cite{FN62}, Theorem 3]) says that the projection map $M^n\to M^r$
  to the first $r$ coordinates restricts to the following locally trivial fiber bundle map,
  with fiber homeomorphic to $PB_{n-r}(\wt M)$, where $\wt M=M-\{r\
\text{points}\}$. 
$$f(M):PB_n(M)\to PB_r(M).$$

Our main motivation for this article is the following corollary of the
fibration theorem. In the rest of the article we assume $r=n-1$.

\begin{thm}([\cite{FN62}, Corollary 2.2])\label{FN} Let $M$ be a connected aspherical $2$-manifold. Then
  $PB_n(M)$ is also aspherical.\end{thm}

\begin{proof} Note that, the fiber of $f(M)$ is aspherical, since
  it is homeomorphic to $M-\{(n-1)\
  \text{points}\}$. Next, since $f(M)$ is a fibration, it induces a long exact sequence of
  homotopy groups. Therefore, by induction on $n$, we conclude that $PB_n(M)$ is aspherical.\end{proof}

It is an important subject to study the
homotopy groups, especially the fundamental groups of the
configuration spaces of a manifold. Since in dimension $\geq 3$, the space
$PB_n(M)$ and the product manifold $M^n$ have isomorphic fundamental groups,
the dimension $2$ case is of much interest. Using the fibration $f(M)$,
there are results to compute the higher homotopy groups of the
configuration space as well. See \cite{FN62} for more on this subject.

Orbifolds are also of fundamental
importance in algebraic and differential geometry, topology and
string theory. In \cite{Rou20} 
we studied the possibility of extending the Fadell-Neuwirth fibration
theorem for orbifolds. However,
to define a fibration between orbifolds, we had to consider the category of
Lie groupoids. Since an orbifold can be realized as a Lie groupoid (\cite{Moe02}),
and there are enough tools in this category to define a fibration.
There, we defined two notions ($a$ and $b$-types)
of a fibration ([\cite{Rou20}, Definition 2.4]) and the corresponding ($a$ and $b$-types)
configuration Lie groupoids of a
Lie groupoid to enable us to state a Fadell-Neuwirth type fibration theorem.
For an orbifold $M$, the $b$-type configuration Lie groupoid is the correct
model to induce the orbifold structure on $PB_n(M)$. We proved that the
Fadell-Neuwirth fibration theorem extends in this generality, under some strong hypothesis
($c$-groupoid). We will show in Proposition \ref{nonfibration},
that this is the best possible extension. For
this, we will deduce that
the map $f(M)$ is not a $a$(or $b$)-type fibration for the $a$(or
$b$)-type configuration
Lie groupoids of Lie groupoids, corresponding to global quotient
orbifolds of dimension $\geq 2$ with a homological condition, and
a non-empty singular set. In fact, for $2$-dimensional orbifolds 
with at least one cone point, $f(M)$ does not induce a long exact
sequence of orbifold homotopy groups. See Remark \ref{les} for a more
precise statement.

Recall that, for a connected
aspherical $2$-manifold $M$, by Theorem \ref{FN}, $PB_n(M)$ is an aspherical
manifold. Equivalently, the universal cover of $PB_n(M)$ is a
contractible manifold.

We define
an orbifold $M$ to be {\it aspherical}, if
its universal orbifold cover $\wt M$ is a contractible orbifold, that is, if the
orbifold homotopy groups $\pi_q^{orb}(\wt M)$, for $q\geq 1$, are trivial.

At this point, one may ask if Theorem \ref{FN} is true for connected $2$-dimensional
orbifolds. In the following theorem we give an answer to this question,
for a class of aspherical $2$-dimensional
orbifolds. Let ${\Bbb C}^*$ be ${\Bbb C}$ with one puncture.

\begin{thm}\label{mt} Let $M$ be one of the following $2$-dimensional
  orbifolds.

  \begin{itemize}

    \item Underlying space is $\Bbb C$ with one cone point of order $m\geq 2$.
    
    \item Underlying space is $\Bbb C$ with two cone points of order
      $2$.
    \item Underlying space is ${\Bbb C}^*$ with one cone point of order $2$.
    \end{itemize}

    Then $PB_n(M)$ is aspherical.\end{thm}

  Here, we make some remark on the reasons behind the three cases in the theorem.
  The first one relates to the solution of the $K(\pi, 1)$-problem of the hyperplane arrangement
  complement corresponding to the unitary reflection groups $G(m,l,n)$ (\cite{TN83}), and the finite type
  Artin groups of type $D_n$ ($m=2$ case) (\cite{Bri73}). The second
  and third cases need the solutions of the $K(\pi, 1)$-problem for the affine Artin groups of types
  $\wt D_n$ and $\wt B_n$, respectively (\cite{PS20}, \cite{CMS10}). The $K(\pi, 1)$-problem for
  hyperplane arrangement complements is well known, and there are milestone results in the literature. See
  \cite{Del} for finite type Artin groups, \cite{PS20} for affine Artin groups, and \cite{DB15}
  for finite complex reflection groups.
  
  We predict the following.

  \medskip
  \noindent
  {\bf Asphericity conjecture.} {\it $PB_n(M)$ is aspherical, for a connected aspherical
    $2$-dimensional orbifold $M$.}

  \medskip
  Since a contractible orbifold is a manifold (\cite{Lyt13}), 
  the universal orbifold cover $\wt M$ of a connected aspherical
  $2$-dimensional orbifold, is an aspherical simply
  connected $2$-manifold. Therefore, $\wt M$ is diffeomorphic to a submanifold of 
  ${\Bbb R}^2$, and $\pi_1^{orb}(M)$ is acting effectively and properly discontinuously
  on $\wt M$. We will prove in Proposition \ref{induction}, that the Asphericity
  conjecture is equivalent to the following statement.

  \medskip
  \noindent
  {\bf Asphericity conjecture for orbit configuration spaces.} {\it Let $H$ be a discrete group
    acting effectively and properly discontinuously on a connected and
    simply connected submanifold $\wt M$ of 
    ${\Bbb R}^2$. Then, the following orbit configuration space
    of the action of $H$ on $\wt M$, is aspherical.
$$PB_n(\wt M, H):=\{(x_1,x_2,\dots, x_n)\in M^n\ |\ Hx_i\neq Hx_j,\ \text{for}\ 
i\neq j\}.$$}

To define the orbifold homotopy groups of an orbifold, one needs to
look at orbifolds as Lie groupoids. In the next section we will see how to study orbifolds in the
    category of Lie groupoids. Then, we will formulate the above
    statements in this
    general setting. Also, we
    will justify the equivalence of the Asphericity conjecture and its orbit
    configuration space version. In the final section we give the
    proof of Theorem \ref{mt}.

\section{Lie groupoids and orbifolds}\label{olgqf}
In this paper by a `manifold' we mean a `Hausdorff smooth manifold' and by a `group' we mean
a `discrete group', unless mentioned otherwise. A `map' is either continuous or smooth,
which will be clear from the context. And a `fibration' would mean 
a `Hurewicz fibration'.

We now recall some basics on Lie groupoids and orbifolds. See  
\cite{ALR07}, \cite{Moe02} or \cite{MP97} for more details.

\subsection{Lie groupoids}\label{lg}
Let $\cal G$ be a {\it Lie groupoid}  with object space ${\cal G}_0$ and morphism
space ${\cal G}_1$. Let $s,t:{\cal G}_1\to {\cal G}_0$ be the source
and the target maps defined by $s(\sigma)=x$ and $t(\sigma)=y$, for
$\sigma\in mor_{\cal G}(x,y)\subset {\cal G}_1$. Recall that
$s$ and $t$ are smooth and submersions.

A {\it homomorphism} $f$ between two Lie groupoids is a smooth functor which respects all the
structure maps. $f_0$ and $f_1$ denote the object and morphism level maps of $f$, 
respectively.  
For any $x\in {\cal G}_0$, the set $t(s^{-1}(x))$ is called the {\it orbit} of $x$. The space
$|{\cal G}|$ of all orbits with respect to the quotient topology is called the
{\it orbit space} of the Lie groupoid. If $f:{\cal G}\to {\cal H}$ is a homomorphism
between two Lie groupoids, then $f$
induces a map $|f|:|{\cal G}|\to |{\cal H}|$, making the following diagram commutative.
We define $\cal G$ to be {\it Hausdorff} if $|{\cal G}|$ is Hausdorff,
and it is called a $c-groupoid$ if the quotient map ${\cal G}_0\to |{\cal G}|$ is a covering map.
Hence a $c$-groupoid is Hausdorff.

\centerline{
  \xymatrix{{\cal G}_0\ar[r]^f\ar[d]&{\cal H}_0\ar[d]\\
    |{\cal G}|\ar[r]^{|f|}&|{\cal H}|.}}

Given a Hausdorff Lie groupoid $\cal G$, we defined in [\cite{Rou20},
Definition 2.8] the {\it b-configuration Lie groupoid}
$PB_n^b({\cal G})$. In this paper we do not use the superscript $b$ as we consider only this
configuration Lie-groupoid. Recall that, its object space $PB_n({\cal G})_0$ is the $n$-tuple of objects
of $\cal G$ with mutually distinct orbits.
$$PB_n({\cal G})_0=\{(x_1,x_2,\ldots , x_n)\in {\cal G}_0^n\ |\ t(s^{-1}(x_i))\neq t(s^{-1}(x_j)),\ \text{for}\ i\neq j\}.$$
The morphism space $PB_n({\cal G})_1$ is $(s^n, t^n)^{-1}(PB_n({\cal G})_0\times PB_n({\cal G})_0)$.
We also showed in [\cite{Rou20}, Lemma 2.9] that the projection to the first $n-1$ coordinates
on both the object and morphism spaces define a
homomorphism $f({\cal G}):PB_n({\cal G})\to PB_{n-1}({\cal G})$.

Let $\cal H$ and $\cal G$ be two Lie
    groupoids and $f:{\cal H}\to {\cal G}$ be a homomorphism, such
    that $f_0:{\cal H}_0\to {\cal G}_0$ is a covering map. Then, $f$ is called a {\it
      covering homomorphism} of Lie groupoids if  ${\cal H}_0$ is a
    left $\cal G$-space with $f_0$ equal to the action map,
    ${\cal H}_1={\cal G}_1\times_{{\cal G}_0}{\cal H}_0$ and $f_1$ is the first projection. The source and the target
  maps of $\cal H$ coincide with the second projection and the
  action map, respectively. 

Next we recall the important concept of the classifying space of a Lie groupoid, which
is required to define algebraic invariants of the Lie groupoid. For a
Lie groupoid $\cal G$, the classifying space $B{\cal G}$ is
defined as the geometric realization of the simplicial manifold
${\cal G}_{\bullet}$ defined by the following iterated fibered products.
$${\cal G}_k={\cal G}_1\times_{{\cal G}_0}{\cal G}_1\times_{{\cal G}_0}\dots \times_{{\cal G}_0}{\cal G}_1.$$

See [\cite{ALR07}, p. 25] or \cite{Moe02} for some
  discussion on this matter, in the context of orbifold Lie groupoids (Example \ref{tg}).

\begin{defn}\label{HG}{\rm The $k$-th {\it homotopy group} of $\cal G$ is defined as the
$k$-th ordinary homotopy group of
$B\cal G$. That is, $\pi_k({\cal G}, \wt x_0):=\pi_k(B{\cal G}, \wt x_0)$ for $\wt x_0\in {\cal G}_0$.
A Lie groupoid $\cal G$ is called {\it aspherical} if $\pi_k({\cal G}, \wt x_0)=0$ for
all $k\geq 2$ and for all $\wt x_0\in {\cal G}_0$.}\end{defn}

Note that, a homomorphism $f:{\cal G}\to {\cal H}$ induces a map $Bf:B{\cal G}\to B{\cal H}$.
Also see \cite{ALR07} or \cite{Moe02} for some more on homotopy groups of Lie groupoids.

Now, we recall the concept of
equivalence between two Lie groupoids. One consequence
of this concept is that an equivalence $f:{\cal G}\to {\cal H}$ between two 
Lie groupoids induces a weak homotopy equivalence $Bf:B{\cal G}\to B{\cal H}$.

A more appropriate notion of equivalence between Lie groupoids
is Morita equivalence.

\begin{defn}\label{moritadef}{\rm Let $f:{\cal G}\to {\cal H}$ be a homomorphism between Lie
    groupoids. $f$ is called an {\it equivalence} if the following conditions are satisfied.

    \begin{itemize}
    \item The following composition is a surjective submersion.

      \centerline{
        \xymatrix{{\cal H}_1\times_{{\cal H}_0}{\cal G}_0\ar[r]^{\ \ \ \ \pi_1} &{\cal H}_1\ar[r]^t &{\cal H}_0.}}
      \noindent
      Here ${\cal H}_1\times_{{\cal H}_0}{\cal G}_0$ is the fibered product,
      defined by $s$ and $f_0$.

      \centerline{
        \xymatrix{{\cal H}_1\times_{{\cal G}_0}{\cal G}_0\ar[r]^{\ \ \ \pi_1}\ar[d]^{\pi_2}&{\cal H}_1\ar[d]^s\\
          {\cal G}_0\ar[r]^{f_0}&{\cal H}_0.}}

    \item The following commutative diagram is a fibered product of manifolds.

      \centerline{
        \xymatrix@C+1pc{{\cal G}_1\ar[r]^{f_1}\ar[d]^{(s,t)}& {\cal H}_1\ar[d]^{(s,t)}\\
          {\cal G}_0\times {\cal G}_0\ar[r]^{\!\!\!{f_0\times f_0}}&{\cal H}_0\times {\cal H}_0.}}
    \end{itemize}

  ${\cal G}$ and ${\cal H}$ are called {\it Morita equivalent} if
there is a third Lie groupoid $\cal K$ and two equivalences as follows.

\centerline{
  \xymatrix {{\cal G}&\ar[l]{\cal K}\ar[r]&{\cal H}.}}}\end{defn}

Next, we recall a standard example of a Lie groupoid which is relevant for us.

\begin{exm}\label{tg}{\rm Let $\wt M$ be a manifold, and a Lie group $H$ is  
    acting on $\wt M$ smoothly. Out of this information one constructs a Lie groupoid
    ${\cal G}(\wt M,H)$ as follows, and call it 
    the {\it translation Lie groupoid}. Define
    ${\cal G}(\wt M,H)_0=\wt M$, ${\cal G}(\wt M,H)_1=\wt M\times H$,
    $s(x,h)=x$, $t(x,h)=h(x)$, $u(x)=(x,1)$, $i(x,h)=(x, h^{-1})$ and $(h(x), h')\circ (x,h)=(x, h'h)$,
    for $h,h'\in H$ and $x\in \wt M$.
    When $H$ is the trivial group then ${\cal G}(\wt M, H)$ is called
    the {\it unit groupoid}, denoted by ${\cal G}(\wt M)$ and is identified with $\wt M$.
    In this paper we always consider $H$ to be discrete (unless explicitly mentioned)
    and is acting effectively and properly discontinuously on $\wt M$ ([\cite{Thu91}, Definition 3.7.1]).
  Then ${\cal G}(\wt M,H)$ is an example of an {\it (effective) orbifold Lie groupoid}. 
  In this case ${\cal G}(\wt M,H)$ is also called an orbifold Lie
  groupoid inducing the orbifold structure on 
  $M=\wt M/H$. 
  For the more general definition of  {\it orbifold Lie groupoid} see
  \cite{Moe02} or  [\cite{ALR07}, Definition 1.38].}\end{exm}

\begin{defn}\label{tolg} {\rm We call an effective orbifold Lie groupoid of
     type 
  ${\cal G}(\wt M, H)$ as in Example \ref{tg}, a {\it translation orbifold Lie
  groupoid}.}\end{defn}

\begin{exm}\label{quotient}{\rm Let $H$ and $\wt M$ be as in Example \ref{tg}.
Let $H'$ be a subgroup of $H$ and $i:H'\to H$ be the inclusion map. Then the maps
    $f_0:=id:\wt M\to \wt M$ and $f_1:=(id, i):\wt M\times H'\to \wt M\times H$ together
    define a homomorphism $f:{\cal G}(\wt M, H')\to {\cal G}(\wt M, H)$.}\end{exm}

Frequently, in this paper we will be using the following lemma.

\begin{lemma}\label{covering} Let $\cal G$ and $\cal H$  
  be two orbifold Lie groupoids, and $f:{\cal G}\to {\cal H}$ be a
  covering homomorphism. Then $f$ induces isomorphisms on higher homotopy groups
and an injection on the fundamental groups.\end{lemma}

\begin{proof} It is easy to see that a covering homomorphism between two orbifold
  Lie groupoids induces a
  covering map on their classifying spaces. The Lemma now follows from
  standard covering space theory. See [\cite{ALR07}, Proposition 2.17].\end{proof}

    \subsection{Orbifolds as Lie groupoids}\label{oalg}
    An {\it orbifold} (\cite{Thu91}) or a {\it $V$-manifold} as in \cite{Sat56},
    is defined as follows.

\begin{defn}\label{orbifold-defn}{\rm Let $M$ be a paracompact
    Hausdorff topological space. Assume for each $x\in M$, there is a connected open
    neighborhood $U_x\subset M$ of $x$ satisfying the following conditions.

    $\bullet$ There is a connected open set $\wt U_x$ in some ${\Bbb R}^n$ and a finite
      group $G_x$ of diffeomorphisms of $\wt U_x$. Furthermore, there is a
      $G_x$-equivariant map $\phi_x:\wt U_x\to M$ such that the induced map 
      $[\phi_x]:\wt U_x/G_x\to M$ is a homeomorphism onto $U_x$.

      Then, $(\wt U_x, G_x, \phi_x)$ is called
      a {\it chart} and $M$ is called an {\it orbifold},
    with {\it underlying space} $M$. Given a
    chart $(U_x, G_x, \phi_x)$, the group $G_x$ can be
    shown to be unique, and is called the {\it local group} at $x$. If the local group
    at $x$ is trivial, then $x$ is called a {\it smooth} or a {\it regular} point, otherwise
    it is a {\it singular} point or a {\it singularity}.}\end{defn}

Assume dimension of $M$ is $2$. 
    Let $(U_x, G_x, \phi_x)$ be a chart such that, $G_x$ is  
    finite cyclic of order $q$, acting by rotation around the origin
    $(0,0)\in \wt U_x\subset {\Bbb R}^2$ by an angle $\frac{2\pi}{q}$
    and $\phi_x((0,0))=x$. Then,
    $x$ is called a {\it cone point} of {\it order} $q$. Also, there are
    two other types of singularities, called {\it reflector lines} and {\it corner
      reflectors}. In this dimension, it is known
    that the underlying space is homeomorphic to a $2$-dimensional manifold.
    The genus of the underlying space $M$ is called
    the {\it genus} of the orbifold $M$. See
    \cite{Sco83} for more details.

    \begin{exm}\label{nice}{\rm An orbifold Lie groupoid $\cal G$ gives
        an orbifold 
      structure on $|{\cal G}|$.
      See [\cite{ALR07}, Proposition 1.44].}\end{exm}

We now have the following useful lemma.

\begin{lemma}\label{covering-lemma} Let ${\cal G}$ and $\cal H$ be two orbifold
  Lie groupoids and $f:{\cal G}\to {\cal H}$ be a homomorphism.
  Then, $f$ is a covering homomorphism of Lie groupoids if and only
  if $|f|:|{\cal G}|\to |{\cal H}|$
is an orbifold covering map.\end{lemma}

\begin{proof} See [\cite{ALR07}, p. 40] for a proof.\end{proof}

\begin{exm}\label{orbifold}{\rm If a group $H$ acts effectively and properly
    discontinuously on a manifold $\wt M$, and $H'$ is a subgroup of $H$, then the  
    homomorphism ${\cal G}(\wt M, H')\to {\cal G}(\wt M, H)$
    (Example \ref{quotient}) is a covering homomorphism.
  In particular, if a finite group $H$ acts effectively
  on a manifold $\wt M$, then ${\cal G}(\wt M)\to {\cal G}(\wt M, H)$
  is a covering homomorphism.}\end{exm}

It is well known that two orbifold Lie groupoids induce 
equivalent orbifold structures on $M$ if and only if they are Morita equivalent
(\cite{MP97}, \cite{Moe02}).

In our situation of translation orbifold Lie groupoids we see in the
following lemma, that when two translation orbifold Lie groupoids are
Morita equivalent then in the Morita equivalence, the third orbifold Lie groupoid also
can be chosen to be a translation orbifold Lie groupoid. We need this
lemma for the proof of Theorem \ref{fmt}.

\begin{lemma}\label{morita} Let two translation orbifold Lie
groupoids ${\cal G}(M_1, H_1)$ and ${\cal G}(M_2, H_2)$
are inducing equivalent orbifold structures on $M$. Then there is a
third translation orbifold Lie groupoid ${\cal G}(M_3, H_3)$, which is equivalent to
both ${\cal G}(M_1, H_1)$ and ${\cal G}(M_2, H_2)$.\end{lemma}

\begin{proof}
Let $p_1:(M_1, m_1)\to (M, m)$ and
$p_2:(M_2, m_2)\to (M, m)$ be the orbifold covering projections, with groups of covering
transformations $H_1$ and $H_2$, respectively. Here $m\in M$ is a smooth point.
Consider the orbifold covering
$M_3$ of $M$ corresponding to the subgroup
$$K:=(p_1)_*(\pi_1(M_1,m_1))\cap (p_2)_*(\pi_1(M_2, m_2)) < \pi_1^{orb}(M,m).$$
Let $H_3$ be
the group of covering transformation of the orbifold covering map $p_3:M_3\to M$.

Then, we will establish the following diagram. 

\centerline{
  \xymatrix{{\cal G}(M_1, H_1)&\ar[l]{\cal G}(M_3, H_3)\ar[r]&{\cal G}(M_2, H_2).}}

\noindent
We will first define the arrows and then show that they are, in
fact, homomorphisms and equivalences.

We just need to check it for one of these arrows, since the
same proof will work for the other one as well.

Denote ${\cal G}(M_3, H_3)$ by $\cal G$ and ${\cal G}(M_1, H_1)$ by
$\cal H$. Let $p:(M_3,m_3)\to (M_1, m_1)$ be the
covering map corresponding to the subgroup $(p_1)_*^{-1}(K)$ of
$\pi_1(M_1, m_1)$, and let 

\centerline{
  \xymatrix{
    \rho:H_3={\pi_1^{orb}(M, m)}/p_{3*}({\pi_1(M_3, m_3)})\ar[r] &{\pi_1^{orb}(M, m)}/p_{1*}({\pi_1(M_1, m_1)})=H_1}}
\noindent
be the quotient homomorphism. Note that, $p$ is a genuine covering map of manifolds. 

Then define $f_0=p$ and $f_1=(p, \rho):M_3\times H_3\to M_1\times H_1$. The following
commutative diagram shows that $f:{\cal G}\to {\cal H}$ is a homomorphism, where $h_3\in H_3$.

\centerline{
  \xymatrix{M_3\ar[rr]^{h_3}\ar[d]^p&&M_3\ar[d]^p\\
    M_1\ar[rr]^{\rho(h_3)}\ar[dr]^{p_1}&&M_1\ar[dl]_{p_1}\\
    &M&}}
\medskip

Now we check that $f$ is an equivalence.  See Definition \ref{moritadef}.

    The first condition in the definition of an equivalence says that the composition

    \centerline{
        \xymatrix{{\cal H}_1\times_{{\cal H}_0}{\cal G}_0\ar[r]^{\ \ \ \ \pi_1} &{\cal H}_1\ar[r]^t &{\cal H}_0}}
\noindent
      should be a surjective submersion. In our situation it takes the following form.
      
    \centerline{
      \xymatrix{
        (M_1\times H_1)\times_{M_1}M_3\ar[r]^{\ \ \ \ \ \pi_1}&(M_1\times H_1)\ar[r]^t& M_1.}}
    \noindent
    The fibered product is with respect to $s$ (first projection) and $p$.
    Since $p$ and $t$ are both surjective submersions, it
        follows that $t\circ\pi_1$ is a surjective submersion.

        Next, we have to check the second condition in the definition of an equivalence,
        that is, we have to show that the following diagram is a fibered product in the
        category of manifolds.

        \centerline{
          \xymatrix@C+1pc{{\cal G}_1\ar[r]^{f_1}\ar[d]^{(s,t)}& {\cal H}_1\ar[d]^{(s,t)}\\
            {\cal G}_0\times {\cal G}_0\ar[r]^{\!\!\!{f_0\times f_0}}&{\cal H}_0\times {\cal H}_0.}}

        In this case, the diagram takes the form.
        
        \centerline{
          \xymatrix{M_3\times H_3\ar[r]^{(p, \rho)}\ar[d]^{(s,t)}&M_1\times H_1\ar[d]^{(s,t)}\\
            M_3\times M_3\ar[r]^{\!\!\!(p,p)}&M_1\times M_1.}}
        \noindent
        Recall that, here $(s,t)(m_3, h_3)=(m_3, h_3(m_3))$ for $(m_3, h_3)\in M_3\times H_3$, and similarly, 
        $(s,t)(m_1, h_1)=(m_1, h_1(m_1))$ for $(m_1, h_1)\in M_1\times H_1$.

        To check that the above diagram is a fibered product, we have to complete the
        following commutative diagram by defining the
        dashed arrow $l$, such that the whole diagram commutes.
         Here $A$ is a manifold
        and the maps $i$, $j$ and $k$ are smooth. The two lower right hand side
        triangles are given to be commutative, that is,
        $(s,t)\circ i=k=(p,p)\circ j$. 
For $a\in A$ let $i(a)=(m_1(a), h_1(a))$ and $j(a)=(m_3(a), m_3'(a))$. Then, we get 
$k(a)=(p(m_3(a)), p(m_3'(a)))=(m_1(a), h_1(a)(m_1(a)))$.

        \centerline{
          \xymatrix{M_3\times H_3\ar[rr]^{(p, \rho)}\ar[dd]^{(s,t)}&&M_1\times H_1\ar[dd]^{(s,t)}\\
            &A\ar [ru]^i\ar[rd]^k\ar[ld]_j\ar@{-->}[lu]_l&\\
            M_3\times M_3\ar[rr]^{(p,p)}&&M_1\times M_1.}}
        
Next, consider the following diagram. The unique map $h_3(a)$ is obtained by
        lifting the composition of the horizontal maps using lifting
        criterion of covering space theory.
Now, we define $l(a)=(m_3(a), h_3(a))$. This completes the proof that $f$ is an equivalence of
        Lie groupoids.

        \centerline{
          \xymatrix@C+1pc{&&(M_3, m_3'(a))\ar[d]^p\\
            (M_3,m_3(a))\ar[r]^p\ar@{-->}[rru]^{h_3(a)}&(M_1,m_1(a))\ar[r]^{\!\!\! h_1(a)}&(M_1,p(m_3'(a))).}}

        This completes the proof of the lemma.\end{proof}

We are now ready to recall the following definition.
  
  \begin{defn}\label{aspherical}{\rm Let $\cal G$ be
      an orbifold Lie groupoid inducing an orbifold structure on $M$.
      Then, the $k$-th {\it orbifold homotopy group}
      $\pi_k^{orb}(M, x_0)$ of $M$ is defined as the homotopy group $\pi_k({\cal G}, \wt x_0)$
      for some $\wt x_0\in {\cal G}_0$ lying above a base point $x_0\in M$. $M$  is called
      {\it aspherical} if $\cal G$ is aspherical.}\end{defn}

    The fundamental group of an orbifold Lie groupoid ${\cal G}$ inducing the
    orbifold structure on $M$, 
is also called the {\it orbifold fundamental group} of the orbifold $M$.
This group is identified with the standard definition
of orbifold fundamental group of $M$ (see \cite{Thu91}). Hence, by
Lemmas \ref{covering} and \ref{covering-lemma}, we get the following lemma.

\begin{lemma}\label{orbcov} An orbifold covering map
induces isomorphisms on higher homotopy groups, and an injection on orbifold
fundamental groups.\end{lemma}

An useful immediate corollary of the lemma is the following.

\begin{cor}\label{orbcovcor} Let $M$ be a connected orbifold and $p:\wt M\to M$ be an  
  orbifold covering map. Assume that $\wt M$ is connected and has no
  singular points.
  Then $p_*:\pi_k(\wt M)\to \pi_k^{orb}(M)$ is
  an isomorphism for all $k\geq 2$.\end{cor}
  
In general an orbifold need not be the quotient of a manifold by
an effective and properly discontinuous
action of a discrete group. For example, the sphere with one cone point and the
sphere with two cone points of different orders are examples of closed $2$-dimensional
orbifolds, which are not covered by manifolds. See \cite{Thu91}. Also see
Proposition 1.54 and Conjecture 1.55 in \cite{ALR07} for some more general
discussion.

It is standard to call a connected orbifold $M$ {\it good}
    or {\it developable} if there is a manifold
    $\wt M$ and an orbifold covering map $\wt M\to M$.

  We will need the following extension of the main theorem of \cite{Rou19}, which says that 
  a connected $2$-dimensional orbifold with finitely generated and infinite orbifold
  fundamental group is good.

  \begin{prop}\label{infgen} Let $M$ be a connected $2$-dimensional orbifold, with infinite
    orbifold fundamental group. Then $M$ is good.\end{prop}

  \begin{proof} First, recall that a $2$-dimensional orbifold has three different types of singular sets:
    cone points, reflector lines and corner reflectors (See [\cite{Sco83}, p. 422]).
    The points on the reflector lines and corner reflectors contributes to the
    boundary of the underlying space, called {\it orbifold boundary}.
    Hence, after taking a double of the underlying
    space along this orbifold  boundary components, and then applying
    [\cite{Rou19}, Lemma 2.1], we get an
    orbifold double cover of $M$ which has only cone points.

    Therefore, we can assume that the orbifold $M$ has only cone points. Also,
    we replace each manifold boundary component with a puncture. Clearly,
    this does not affect the orbifold fundamental group of $M$.
    
    Note that, if $\pi_1^{orb}(M)$ is finitely generated, then the proposition
    follows from [\cite{Rou19}, Theorem 1.1]. If it is infinitely
    generated, then $M$ either has infinite genus or has infinitely
    many punctures or infinitely many cone points.

    Hence, we can write $M$ as an infinite increasing union of orbifolds of the type
    $M(r_i, s_i)$, $i\in {\Bbb N}$. Each $M(r_i, s_i)$ has finite genus, $r_i$ number of
    punctures and $s_i$ number of cone points. Furthermore, we can assume that $M(r_i, s_i)$ has
    infinite orbifold fundamental group. Then clearly, $M(r_i, s_i)$ is aspherical,
    since they are all good orbifolds with infinite orbifold fundamental groups. Hence, a direct
    limit argument shows that $M$ is also aspherical,
    since the orbifold homotopy groups are covariant
    functors. Therefore, by Lemma \ref{orbcov}, the universal orbifold covering 
    $\wt M$ of $M$ has all the orbifold homotopy groups trivial.
    We now apply \cite{Lyt13} to conclude that $\wt M$ is a manifold. Hence $M$
    is a good orbifold.
  \end{proof}

  \section{Asphericity}\label{qfa}
In this section we 
state the Asphericity conjecture and Theorem \ref{mt},  
in the general set up of the category of Lie groupoids.

If $M$ is an orbifold, then $PB_n(M)$ is an orbifold, since it is
an open set in the product orbifold $M^n$.

Let $H$ be a group acting effectively and properly
discontinuously on a connected manifold $\wt M$. Then the 
quotient $M=\wt M/H$ has an orbifold structure.

Consider the space $PB_n(\wt M, H)$ of $n$-tuples of points of
$\wt M$ with pairwise distinct orbits, defined in the Introduction.
Then, $H^n$ acts effectively and properly
discontinuously on $PB_n(\wt M, H)$, with quotient, the orbifold $PB_n(M)$.
Hence, $PB_n({\cal G}({\wt M}, H))$, the configuration Lie groupoid
of $n$ points of ${\cal G}(\wt M,H)$ is the corresponding
translation orbifold Lie groupoid ${\cal G}(PB_n(\wt M,H), H^n)$. 
Recall that, for a good orbifold $M$, we can have many
configuration Lie groupoids associated
to different regular orbifold coverings of $M$. 
Also, clearly given two such orbifold coverings, the corresponding
configuration Lie groupoids will induce equivalent orbifold structures on $PB_n(M)$.

In [\cite{Rou20}, Theorem 2.10] we proved that the homomorphism
$f({\cal G}):PB_n({\cal G})\to PB_{n-1}({\cal G})$
is a $b$-fibration for a $c$-groupoid $\cal G$. We had also shown
that, this is the best possible class of Lie groupoids to which this
fibration result can be proven. In the following proposition, we
deduce a more general statement.

A surjective map $g:X\to Y$ is called a {\it quasifibration},
if $g:(X, g^{-1}(y))\to (Y,y)$ is a weak homotopy equivalence, for all $y\in Y$ (\cite{DT58}).
Hence, a fibration is a quasifibration, and also a quasifibration
induces a long exact sequence of homotopy groups.

\begin{prop}\label{nonfibration} Let $H$ be a finite group, 
  acting effectively on a connected 
    manifold $\wt M$ of dimension $m\geq 2$, with at least
    one fixed point. Assume that the integral homology group $H_m(\wt M, {\Bbb Z})$ is
    finitely generated. Then the projection map $f_0:PB_n(\wt M, H)\to PB_{n-1}(\wt M, H)$ is not
    a quasifibration.\end{prop}

  \begin{proof} By hypothesis, there is a point $s\in \wt M$, so that
  the isotropy group $H_s=\{h\in H\ |\ hs=s\}$ is nontrivial.  
  There is a neighbourhood $U_s\subset \wt M$ of $s$ preserved by $H_s$ and 
  $hU_s\cap U_s=\emptyset$ for all $h\in H - H_s$. Such a
  neighbourhood exists, see the proof of [\cite{Thu91}, Proposition 5.2.6].
  Since regular points are dense in $\wt M/H$, there is a point $s'\in U_s$
  which has trivial isotropy group. That is, $s'$ corresponds
  to a regular point and $s$ corresponds to a singular point on $\wt M/H$.

   Choose a point $x=(s, x_2,\ldots, x_{n-1})\in PB_{n-1}(\wt M, H)$, such that
  $Hx_i\neq Hs'$ for all $i=2,3,\ldots, n-1$. Let $y=(s', x_2,\ldots, x_{n-1})$.
  Note that $|Hs| < |Hs'|$. 
  Then it follows that $f_0^{-1}(x)$ and $f_0^{-1}(y)$ are obtained from
  $\wt M$, by removing $|Hs|$ and $|Hs'|$ number of points,
  respectively.
  Hence, they have
  non-isomorphic integral homology groups in dimension $m$, since
  $H_m(\wt M, {\Bbb Z})$ is finitely generated. Therefore,  $f_0^{-1}(x)$ and $f_0^{-1}(y)$
  are not weak homotopy equivalent. On the other hand, for a
  quasifibration over a path connected space any two fibers are
  weak homotopy equivalent (\cite{Hat03}, chap. 4, p. 479).
  Therefore, $f_0$ is not a quasifibration.\end{proof}

    The examples in the proposition above are the primary reasons why the Fadell-Neuwirth
  fibration theorem does not extend to orbifolds, and more generally to
  Lie groupoids.

\begin{rem}\label{les}{\rm In \cite{JF23}, Flechsig pointed out that the short exact sequence
    proved in \cite{Rou20} is, in fact, a four-term exact
    sequence. That is, the kernel of the homomorphism
    $f(M)_*:\pi_1^{orb}(PB_n(M))\to \pi_1^{orb}(PB_{n-1}(M))$ is not
    isomorphic to the orbifold fundamental group of a generic fiber
    (that is, fiber over a smooth point) of
    $f(M)$, if $M$ is a genus zero $2$-dimensional orbifold with at
    least one 
    puncture and at least one cone point. See \cite{Rou23} for more
    on this matter. Therefore, together with Theorem \ref{mt}, for
    such $M$ we 
    conclude that $f(M)$ is not a quasifibration of orbifolds, that is,
    it does not induce a long exact sequence of orbifold homotopy
    groups. Furthermore, using [\cite{Rou23}, Theorem 2.2], the same
  would not be true in general also if the Asphericity conjecture has
  a positive answer.}\end{rem}

Let ${\cal C}$ be the class of all connected
$2$-dimensional orbifolds with $\pi_1^{orb}(M)$ infinite.
Let $M\in {\cal C}$.
Then, by Proposition \ref{infgen}, $M$ is a good orbifold. For
convenience we denote by ${\cal G}_M$, a translation 
orbifold Lie groupoid ${\cal G}(\wt M, H)$, inducing
the orbifold structure on $M$. Since $M$ is good, there are
many such translation orbifold Lie groupoids.

The following proposition 
justifies the equivalence between the Asphericity conjecture and its orbit
configuration space version, of the Introduction.

\begin{prop}\label{induction} Let $M\in {\cal C}$ and consider a translation 
  orbifold Lie groupoid ${\cal G}(\wt M, H)$, such that $M=\wt M/H$. Then, $PB_n(M)$
  is aspherical if and only if the manifold $PB_n(\wt M, H)$ is aspherical.\end{prop}

\begin{proof} Note that, by our convention (Example \ref{tg}) $H$ is acting on $\wt M$ effectively and
  properly discontinuously, so that $M=\wt M/H$.
  Therefore, the quotient map $PB_n(\wt M, H)\to PB_n(M)$ is an orbifold
  covering map. Hence by Corollary \ref{orbcov},
  $PB_n(\wt M, H)$ is aspherical if and only if $PB_n(M)$ is aspherical.\end{proof}

\begin{cor} \label{AC} Let $M$ be as in the statement of Theorem \ref{mt}. Then, 
  for any translation orbifold Lie groupoid ${\cal G}_M$,
  $PB_n({\cal G}_M)$ is aspherical.\end{cor}

The above corollary is
equivalent to Theorem \ref{mt}, but stated in the
category of Lie groupoids. The advantage of this statement is that we
can now state the Asphericity conjecture in a wider context.

Recall that, we defined a Lie groupoid $\cal G$ to be 
  Hausdorff if $|{\cal G}|$ is Hausdorff.

  \medskip
  \noindent
  {\bf Asphericity Problem.} {\it Consider an aspherical Hausdorff Lie groupoid
  $\cal G$, such that ${\cal G}_0$ is connected and $2$-dimensional,
  then $PB_n({\cal G})$ is aspherical.}

\medskip

Next, consider the homomorphism
$f({\cal G}_M):PB_n({\cal G}_M)\to PB_{n-1}({\cal G}_M)$,  for $M\in {\cal C}$.

We end this section by giving a functorial relationship between the homomorphisms
$f({\cal G}_M)_*$ and $f({\cal H}_M)_*$, for
two translation orbifold Lie groupoids ${\cal G}_M$ and ${\cal H}_M$, respectively.

\begin{thm}\label{fmt} 
  For $M\in {\cal C}$ and for any two 
translation orbifold Lie groupoids ${\cal G}_M$ and ${\cal H}_M$, we have the following
commutative diagram for all $q$, where the
  horizontals maps are isomorphisms.
  
\centerline{
  \xymatrix{\pi_q(PB_n({\cal G}_M))\ar[d]^{{f({\cal G}_M)}_*}\ar[r]&
    \pi_q(PB_n({\cal H}_M))\ar[d]^{{f({\cal H}_M)}_*}\\
\pi_q(PB_{n-1}({\cal G}_M))\ar[r]&\pi_q(PB_{n-1}({\cal H}_M)).}}
\end{thm}
    
    \begin{proof}
The statement is to relate the identification of homotopy groups
of the different translation orbifold Lie groupoids
inducing the orbifold structure on $PB_n(M)$, and the corresponding homomorphism
$f({\cal G}_M)_*$, via a commutative diagram.

Since the two translation orbifold Lie groupoids ${\cal G}_M$ and ${\cal H}_M$ 
induce the same orbifold structure on $M$, by Lemma
\ref{morita}
there is another translation orbifold Lie groupoid ${\cal K}_M$ and
a diagram of equivalences.

\centerline{
  \xymatrix{{\cal G}_M&\ar[l]{\cal K}_M\ar[r]&{\cal H}_M.}}

\noindent
Since all the orbifold Lie groupoids we are considering are of
translation type, it is easy to see that the above diagram induces
the following diagram of equivalences. For, the morphism space of
$PB_n({\cal G}(\wt M, H))$ is nothing but $PB_n({\cal G}(\wt M,H))_0\times
H^n=PB_n(\wt M,H)\times H^n$, for any translation orbifold Lie groupoid ${\cal G}(\wt M,H)$. 
See Definition \ref{moritadef}.

\centerline{
  \xymatrix{PB_n({\cal G}_M)&\ar[l]PB_n({\cal K}_M)\ar[r]&PB_n({\cal H}_M).}}

\noindent
Hence, we get the following commutative diagram
of homomorphisms.

\centerline{
  \xymatrix{PB_n({\cal G}_M)\ar[d]^{f({\cal G}_M)}&
    \ar[l]PB_n({\cal K}_M)\ar[r]\ar[d]^{f({\cal K}_M)}&
    PB_n({\cal H}_M)\ar[d]^{f({\cal H}_M)}\\
PB_{n-1}({\cal G}_M)&\ar[l]PB_{n-1}({\cal K}_M)\ar[r]&PB_{n-1(}{\cal H}_M).}}

\noindent
In the diagram all the horizontal homomorphisms are equivalences, and hence
induce weak homotopy equivalences on the classifying spaces.

Now, note that the homomorphisms $f({\cal G}_M)$, 
$f({\cal K}_M)$ and $f({\cal H}_M)$ all induce the same projection
map $PB_n(M)\to PB_{n-1}(M)$.

The theorem now follows, by applying the homotopy functor on the
above diagram.\end{proof}

\section{Proof}
In this section we prove Theorem \ref{mt}, which supports the Asphericity conjecture.

\begin{proof}[Proof of Theorem \ref{mt}] {\bf Case 1.}
Let $M$ be the complex plane
with a cone point of
order $m\geq 2$ at the origin. That is, the underlying space of
$M$ is $\Bbb C$ and the cyclic group (say $H$) of order $m$ is acting on $\Bbb C$ by
rotation about the origin, by an angle $\frac{2\pi}{m}$,
to get the orbifold $M$ as the quotient.

First note that, in this case the
universal orbifold cover of $M$ is $\wt M={\Bbb C}$. Then, note that $PB_n(\wt M, H)$ is
  the following hyperplane arrangement complement.

  $$PB_n(\wt M, H)=\{(z_1,z_2,\ldots, z_n)\in {\Bbb C}^n\ |\ z_i^m\neq z_j^m,\ i\neq j\}.$$ 
\noindent  
  In [\cite{TN83}, p5] it was shown that $PB_n(\wt M, H)$
  fibers over the manifold $PB_{n-1}({\Bbb C}^*)$. For $m=2$ it was established in \cite{Bri73}.
  The fibration $B:PB_n(\wt M, H)\to PB_{n-1}({\Bbb C}^*)$ is defined
  by $z_j\mapsto z_n^m-z_j^m$, for $j=1,2,\dots, n-1$. By Theorem \ref{FN},
  $PB_{n-1}({\Bbb C}^*)$ is aspherical. Therefore, again using the long exact sequence of homotopy groups,
  and by an induction on $n$, we get that $PB_n(\wt M, H)$ is aspherical.
  Hence, by Corollary \ref{orbcovcor} $PB_n(M)$ is aspherical.

  \medskip
  \noindent
  {\bf Case 2.} Assume $M$ has $\Bbb C$ as the underlying space, with two
  cone points at $0$ and $\frac{1}{2}$ of order $2$ each.
  
  Consider the hyperplane arrangement
  complement corresponding to the affine Artin group of type $\wt D_n$ (\cite{Bri73}).
$$P\wt D_n({\Bbb C}):=\{(z_1,z_2,\ldots, z_n)\in {\Bbb C}^n\ |\ z_i\pm z_j\notin {\Bbb Z},\ i\neq j\}.$$ 
  We proceed to show that there is an 
  orbifold covering map $P\wt D_n({\Bbb C})\to PB_n(M)$. We define this map in the
  following diagram.
  
  \centerline{
    \xymatrix{P\wt D_n({\Bbb C})\ar[r]^E\ar[dr]&X\ar[d]^Q\\
      &PB_n(M)}}
\noindent
  Here $X=\{(w_1,w_2,\ldots, w_n)\in ({\Bbb C}^*)^n\ |\ z_i\neq z_j^{\pm 1},\ i\neq j\}$.
  The map $E$ is the restriction of the $n$-fold product of the exponential
  map $z\mapsto \exp(2\pi iz)$, and $Q$ is 
  the restriction of the $n$-fold product of the map $q:{\Bbb C}^*\to M$, defined by
  $$q(w)=\frac{1}{4}\left(1-\frac{1+w^2}{2w}\right).$$
  $E$ is a genuine covering map and $Q$ is a $2^n$-sheeted orbifold covering map. Since $q$
  is a $2$-fold orbifold covering map as $q$ sends the branch point $+1$ to $0$ and 
  $-1$ to $\frac{1}{2}$, and it is of degree $2$ around these points. Therefore,
  $Q\circ E:P\wt D_n({\Bbb C})\to PB_n(M)$ is an
  orbifold covering map.
  
  On the other hand recently, in \cite{PS20}, it was proved that
  $P\wt D_n({\Bbb C})$ is aspherical. Therefore, by Corollary \ref{orbcovcor} 
  $PB_n(M)$ is also aspherical.

  \medskip
  \noindent
  {\bf Case 3.} Assume $M$ has ${\Bbb C}-\{1\}$  as the underlying
  space with $0$ a cone point of order $2$.

  Consider the following hyperplane arrangement complement.

$$W=\{w\in{\Bbb C}^n\ |\ w_i\neq \pm w_j,\ \text{for all}\ i\neq j; w_k\neq \pm 1,\ \text{for all}\ k\}.$$

In [\cite{CMS10}, \S 3] the following homeomorphism is observed.

$${\Bbb C}^*\times W\simeq X:=\{x\in{\Bbb C}^{n+1}\ |\ x_i\neq \pm
x_j,\ \text{for all}\ i\neq j; x_1\neq 0\}.$$
$$(\lambda, w_1,w_2,\ldots, w_n)\mapsto (\lambda, \lambda w_1,\ldots, \lambda w_n)$$

In [\cite{CMS10}, Lemma 3.1] it is then proved that the 
hyperplane arrangement complement $X$ is simplicial, 
in the sense of \cite{Del}. Hence, again by \cite{Del} $X$ is aspherical.
Therefore, $W$ is aspherical.

Next, note that there is the following finite sheeted orbifold covering map.
$$W\to PB_n(M).$$
$$(w_1,w_2,\dots,w_n)\mapsto (w_1^2,w_2^2,\dots,w_n^2)$$

Thus, $PB_n(M)$ is also aspherical.

This completes the proof of the theorem.
    \end{proof}

    \begin{rem}{\rm For a direct proof of asphericity of $X$, see \cite{LR24}.}\end{rem}
    
  \newpage
\bibliographystyle{plain}
\ifx\undefined\bysame
\newcommand{\bysame}{\leavevmode\hbox to3em{\hrulefill},}
\fi

\end{document}